\documentclass[english,11pt]{article}

\usepackage{amssymb,amsmath,amsfonts,amsthm,rotating,setspace,graphicx,float}
\usepackage{amssymb,amsmath,amsthm,rotating,setspace}
\usepackage{enumerate}

\usepackage{latexsym}

\usepackage{epsfig}

\usepackage{color}

\newtheorem{theorem}{Theorem}[section]
\newtheorem{lemma}[theorem]{Lemma}

\theoremstyle{definition}

\newtheorem{remark}[theorem]{Remark}

\theoremstyle{proposition}
\newtheorem{proposition}[theorem]{Proposition}

\numberwithin{equation}{section}

\def\eq#1{(\ref{#1})}

\def\R{\mathbb{R}}

\begin{document}

\title{Weak Solutions for Singular Quasilinear Elliptic Systems}

\author{{\large Gurpreet Singh}\\School of Mathematics and Statistics \\
 University College Dublin,\\
 Belfield, Dublin 4, Ireland\\
E-mail: {\tt gurpreet.singh@ucdconnect.ie}}
\date{12 February, 2016}

\maketitle

\begin{abstract}
We investigate the quasilinear elliptic system
\begin{equation}\label{sysS}
\left\{
\begin{aligned}
-\Delta_{m} u&=u^{-p}v^{-q},u>0 &&\quad\mbox{ in } \Omega,\\
-\Delta_{m} v&=u^{r}v^{-s}, v>0 &&\quad\mbox{ in } \Omega,\\
u&=v=0 &&\quad\mbox{ on }\partial{\Omega},
\end{aligned}
\right.
\end{equation}
where $\Omega \subset \R^{N}(N\geq 1)$ is a bounded and smooth domain, $1<m<\infty$, $p, q, r, s>0$. Under certain conditions imposed on the exponents we obtain the existence and uniqueness of a weak solution $(u, v)$ with $u$, $v \in W_{0}^{1, m}(\Omega)\cap C(\Omega)$. We also investigate the $ W_{0}^{1, \tau}(\Omega)$ regularity of solution and determine the optimal range of $\tau \geq m$ for such regularity.
\end{abstract}

\bigskip

\noindent{\bf Keywords:} Singular quasilinear elliptic systems, $m$-Laplace operator, weak solution, singular nonlinearity, regularity in Sobolev space.

\medskip

\noindent{\bf MSC:} 35J92, 35J75, 35J47.


\section{Introduction and  Main Results}
In this paper we are concerned with the system
\begin{equation}\label{sys}
\left\{
\begin{aligned}
-\Delta_{m} u&=u^{-p}v^{-q},u>0 &&\quad\mbox{ in } \Omega,\\
-\Delta_{m} v&=u^{r}v^{-s}, v>0 &&\quad\mbox{ in } \Omega,\\
u&=v=0 &&\quad\mbox{ on }\partial{\Omega},
\end{aligned}
\right.
\end{equation}
where $\Omega \subset \R^{N}(N\geq 1)$ is a bounded and smooth domain, $1<m<\infty$ and $p, q, r, s>0$. Here, $\Delta_{m}$ is the $m$-Laplace operator defined as
$$
\Delta_{m} u= {\rm div}(|\nabla u|^{m-2}\nabla u).
$$
The $m$-Laplace operator is used in mathematical models that arise in chemical reactions, psuedo-plastic flows, population dynamics, astrophysics, morphogenesis and many other applications. We investigate the existence and uniqueness of weak solutions to system \eq{sys}. By a solution of \eq{sys} we understand a pair $(u, v)$ with $u$, $v\in W_{0}^{1, m}(\Omega)\cap C(\Omega)$ that satifies $u$, $v>0$ and
\begin{equation*}
\left\{
\begin{aligned}
\int_{\Omega}|\nabla u|^{m-2}\nabla u \cdot \nabla \phi dx&= \int_{\Omega}u^{-p}v^{-q}\phi dx &&\quad\mbox{ for all }\phi \in C_{c}^{\infty}(\Omega),\\
\int_{\Omega}|\nabla v|^{m-2}\nabla v \cdot \nabla \psi dx&= \int_{\Omega}u^{r}v^{-s}\psi dx &&\quad\mbox{ for all }\psi \in C_{c}^{\infty}(\Omega).
\end{aligned}
\right.
\end{equation*} 
Systems of type \eq{sys} with all the exponents negative have been studied recently by Giacomoni, Schindler and Takac  in \cite{GST2015}. They considered the system
\begin{equation}\label{sys1}
\left\{
\begin{aligned}
-\Delta_{p_1} u&=u^{-a_1}v^{-b_1},u>0 &&\quad\mbox{ in } \Omega,\\
-\Delta_{p_2} v&=u^{-b_2}v^{-a_2}, v>0 &&\quad\mbox{ in } \Omega,\\
u&=v=0 &&\quad\mbox{ on }\partial{\Omega},
\end{aligned}
\right.
\end{equation}
where $\Omega \subset \R^{N}$ is a bounded and smooth domain, $1<p_{1}, p_{2}<\infty$ and $a_{1}, a_{2}, b_{1}, b_{2}>0$ are positive constants. They employed monotonicity methods in order to prove the existence and uniqueness of a positive solution $(u, v)$ to \eq{sys1} and they were able to get the regularity of solution in the H\"{o}lder Space $C^{0, \beta}(\overline \Omega)$ for some $\beta \in (0, 1)$. Systems of type \eq{sys} with all the exponents positive were studied by Clement, Fleckinger, Mitidieri and Thelin in \cite{CFMT2000}. They considered the following system
\begin{equation}\label{psys}
\left\{
\begin{aligned}
-\Delta_{p} u&=u^{a_{1}}v^{b_{1}},u>0 &&\quad\mbox{ in } B_{R},\\
-\Delta_{q} v&=u^{b_2}v^{a_2}, v>0 &&\quad\mbox{ in } B_{R},\\
u&=v=0 &&\quad\mbox{ on }\partial{B_{R}},
\end{aligned}
\right.
\end{equation}
where $p$, $q\geq 1$, $a_{1}, a_{2}, b_{1}, b_{2}\geq 0$ and $B_{R}$ is an open ball in $\R^{N}$. Systems of type \eq{sys} have also been studied in \cite{CT2010, GH2011, GH2013, GST2007}. The singular semilinear case $p_1=p_2=2$ in \eq{sys1} has been studied even more frequently in \cite{BS2014, CM2000, CM2003, G2010, G2011, HMV2008}. For instance, if $m=2, p<0$ and $q, r ,s>0$, system \eq{sys} corresponds to Gierer-Meinhardt system \cite{GM1972} in morphogenesis.

\medskip

The existence of a solution $(u, v)$ to system \eq{sys} is obtained under the following assumption on the exponents $p$, $q$, $r$ and $s$
\begin{equation}\label{eqf}
\frac{qr}{(1+p)(1+s)}<1.
\end{equation}

\medskip

In order to obtain the existence of a solution $(u, v)$ to \eq{sys} we employ the Schauder fixed point theorem in a closed convex subset of $C(\overline{\Omega})\times C(\overline{\Omega})$ which contains all the functions having a certain rate of decay expressed in terms of the distance function up to the boundary of $\Omega$. This will be done by investigating the singular problem
\begin{equation}\label{sys2}
\left\{
\begin{aligned}
-\Delta_{m} u&= K(x)u^{-p}, u>0 &&\quad\mbox{ in }\Omega,\\
u&=0 &&\quad\mbox{ on }\partial{\Omega},
\end{aligned}
\right.
\end{equation}
where $m>1$, $p\geq 0$ and $K: \Omega \rightarrow (0, \infty)$ is a positive function that behaves like $\delta(x)^{-q}$ where $\delta(x)= {\rm dist}(x, \partial{\Omega})$ and $q\geq 0$ satisfies
$$
p\Big(1-\frac{1}{m}\Big)+q<2- \frac{1}{m}.
$$
The existence of a positive weak solution $u$ of \eq{sys2} with
\begin{equation*}
u \in W_{0}^{1, m}(\Omega)\cap C^{1, \alpha}(\overline{\Omega}) \quad\mbox{ if }p+q<1
\end{equation*}
or 
\begin{equation*}
u \in W_{0}^{1, m}(\Omega)\cap C^{1, \beta}(\overline{\Omega}) \quad\mbox{ if }p+q\geq 1
\end{equation*}
for some $\alpha, \beta \in (0, 1)$ has been obtained in \cite{BGH2015, GMS2012, GST2015}. Problems of type \eq{sys2} have been studied in \cite{BBGGPV1995, BO2009, C2013} and the semilinear case $m=2$ has been studied in \cite{CRT1977, GR2003, GRbook,GCLH1993}.

\medskip

In fact, one can say more about the regularity of the solution $u$ of \eq{sys2}. More precisely we obtain in Section 2 of this work that the unique weak solution $u$ of \eq{sys2} belongs to the Sobolev space $W_{0}^{1, \tau}(\Omega)$ where
$$
m\leq \tau< \frac{m+p-1}{p+q-1} \quad\mbox{ if } p+q>1,
$$
and
$$
m\leq \tau< \infty \quad\mbox{ if } p+q=1.
$$
The $W_{0}^{1, \tau}(\Omega)$ regularity we obtained on the above range of $\tau$ is optimal. This result also transfers to the regularity of solution $(u, v)$ of \eq{sys} as we shall see below. 

\medskip

Throughout this paper, for any two functions $f$, $g$ defined on $\Omega$ we shall write $f\sim g$ to denote that
$$
c_{1}\leq \frac{f}{g}\leq c_{2} \quad\mbox{ in }\Omega,
$$
for some positive constants $c_{1}$ and $c_{2}$. 

\medskip

Our first result concerns the existence of a solution to \eq{sys}.

\medskip

\begin{theorem}\label{mr}{\rm ( \bf Existence)}
Assume \eq{eqf}. Let $m>1$ and $p, q, r,s>0$ satisfy
$$
p\Big(1-\frac{1}{m}\Big)+q<2- \frac{1}{m}
$$
and 
\begin{equation*}
s\Big(1-\frac{1}{m}\Big)-r<2- \frac{1}{m}.
\end{equation*}
\begin{itemize}
\item[{\rm (i)}] If $s-r>1$ and $p+\frac{q(m+r)}{m+s-1}<1$, then \eq{sys} has a positive weak solution $(u, v)$ such that $u \in W_{0}^{1, m}(\Omega) \cap C^{1, \alpha}(\overline{\Omega})$ for some $\alpha \in(0, 1)$ and $v \in W_{0}^{1, \tau}(\Omega) \cap C^{0, \beta}(\overline{\Omega})$ for some $\beta \in(0, 1)$ and for all $m \leq \tau \leq \frac{m+s-1}{s-r-1}$. Also
$$
u(x) \sim \delta(x) \quad\mbox{ and }  v(x) \sim \delta(x)^{\frac{m+r}{m+s-1}}.
$$
\item[{\rm (ii)}] If $s-r>1$ and $p+\frac{q(m+r)}{m+s-1}=1$, then \eq{sys} has a positive weak solution $(u, v)$ such that $u \in W_{0}^{1, \tau_{1}}(\Omega) \cap C^{0, \alpha}(\overline{\Omega})$ for some $\alpha \in(0, 1)$ and for all $m \leq \tau_{1}< \infty$ and $v \in W_{0}^{1, \tau_{2}}(\Omega) \cap C^{0, \beta}(\overline{\Omega})$ for some $\beta \in(0, 1)$ and for all $m \leq \tau_{2} \leq \frac{m+s-1}{s-r-1}$. Also
$$
u(x) \sim \delta(x)\log^{\frac{1}{m+p-1}}\Big(\frac{1}{\delta(x)}\Big) \quad\mbox{ and }  v(x) \sim \delta(x)^{\frac{m+r}{m+s-1}}.
$$
\item[{\rm (iii)}] If $s-r<1$ and $p+q<1$, then \eq{sys} has a positive weak solution $(u, v)$ such that $u \in W_{0}^{1, m}(\Omega) \cap C^{1, \alpha}(\overline{\Omega})$ for some $\alpha \in(0, 1)$ and $v \in W_{0}^{1, m}(\Omega) \cap C^{1, \beta}(\overline{\Omega})$ for some $\beta \in(0, 1)$. Also
$$
u(x) \sim \delta(x) \quad\mbox{ and }  v(x) \sim \delta(x).
$$
\item[{\rm (iv)}] If $s-r<1$ and $p+q=1$, then \eq{sys} has a positive weak solution $(u, v)$ such that $u \in W_{0}^{1, \tau}(\Omega) \cap C^{0, \alpha}(\overline{\Omega})$ for some $\alpha \in(0, 1)$ and for all $m\leq \tau< \infty$ and $v \in W_{0}^{1, m}(\Omega) \cap C^{1, \beta}(\overline{\Omega})$ for some $\beta \in(0, 1)$. Also
$$
u(x) \sim \delta(x)\log^{\frac{1}{m+p-1}}\Big(\frac{1}{\delta(x)}\Big) \quad\mbox{ and }  v(x) \sim \delta(x).
$$
\item[{\rm (v)}] If $s-r=1$ and $p+q<1$, then \eq{sys} has a positive weak solution $(u, v)$ such that $u \in W_{0}^{1, m}(\Omega) \cap C^{1, \alpha}(\overline{\Omega})$ for some $\alpha \in(0, 1)$ and $v \in W_{0}^{1, \tau}(\Omega) \cap C^{0, \beta}(\overline{\Omega})$ for some $\beta \in(0, 1)$ and for all $m \leq \tau< \infty$. Also
$$
u(x) \sim \delta(x) \quad\mbox{ and }  v(x) \sim \delta(x)\log^{\frac{1}{m+s-1}}\Big(\frac{1}{\delta(x)}\Big).
$$
\item[{\rm (vi)}] If $s-r=1$ and $p+q=1$, then \eq{sys} has a positive weak solution $(u, v)$ such that $u \in W_{0}^{1, \tau_{1}}(\Omega) \cap C^{0, \alpha}(\overline{\Omega})$ for some $\alpha \in(0, 1)$ and for all $m\leq \tau_{1}< \infty$ and $v \in W_{0}^{1, \tau_{2}}(\Omega) \cap C^{0, \beta}(\overline{\Omega})$ for some $\beta \in(0, 1)$ and for all $m \leq \tau_{2}< \infty$. Also
$$
u(x) \sim \delta(x)\log^{\frac{1}{m+p-1}}\Big(\frac{1}{\delta(x)}\Big) \quad\mbox{ and }  v(x) \sim \delta(x)\log^{\frac{1}{m+s-1}}\Big(\frac{1}{\delta(x)}\Big).
$$
\item[{\rm (vii)}] If $p+q>1$ and $s-\frac{r(m-q)}{m+p-1}<1$, then \eq{sys} has a positive weak solution $(u, v)$ such that $u \in W_{0}^{1, \tau}(\Omega) \cap C^{0, \alpha}(\overline{\Omega})$ for some $\alpha \in(0, 1)$ and for all $m\leq \tau< \frac{m+p-1}{p+q-1}$ and $v \in W_{0}^{1, m}(\Omega) \cap C^{1, \beta}(\overline{\Omega})$ for some $\beta \in(0, 1)$. Also
$$
u(x) \sim \delta(x)^{\frac{m-q}{m+p-1}} \quad\mbox{ and }  v(x) \sim \delta(x).
$$
\item[{\rm (viii)}] If $p+q>1$ and $s-\frac{r(m-q)}{m+p-1}=1$, then \eq{sys} has a positive weak solution $(u, v)$ such that $u \in W_{0}^{1, \tau_{1}}(\Omega) \cap C^{0, \alpha}(\overline{\Omega})$ for some $\alpha \in(0, 1)$ and for all $m\leq \tau_{1}< \frac{m+p-1}{p+q-1}$ and $v \in W_{0}^{1, \tau_{2}}(\Omega) \cap C^{0, \beta}(\overline{\Omega})$ for some $\beta \in(0, 1)$ and for all $m\leq \tau_{2}< \infty$. Also
$$
u(x) \sim \delta(x)^{\frac{m-q}{m+p-1}} \quad\mbox{ and }  v(x) \sim \delta(x)\log^{\frac{1}{m+s-1}}\Big(\frac{1}{\delta(x)}\Big).
$$
\end{itemize}

\end{theorem}

\medskip

In our next result we discuss the uniqueness of a weak solution to \eq{sys}.

\medskip

\begin{theorem}\label{mru}{\rm (\bf Uniqueness)}
Let $m>1$ and assume \eq{eqf}. If one of the following holds:
\begin{itemize}
\item[{\rm (i)}] $s-r>1$ and $p+\frac{q(m+r)}{m+s-1}< 1$;
\item[{\rm (ii)}] $p+q< 1$ and $s-r\leq 1$;
\item[{\rm (iii)}] $p+q= 1$ and $s-r< 1$;
\item[{\rm (iv)}] $p+q>1$ and $s-\frac{r(m-q)}{m+p-1}<1$;
\end{itemize}
Then system \eq{sys} has a unique solution.
\end{theorem}

The rest of the paper is organized as follows. In Section 2 we collect some preliminary results which include several properties of the solution to the singular problem \eq{sys2}. Finally in Section 3 and Section 4 we give the proof of the main results.

\section{Preliminary results}
In this section we collect some preliminary results that will be useful in the study of \eq{sys}
\begin{proposition}\label{gst}{\rm (see \cite[Theorem 3.1]{GST2015}).}
Let $1< m< \infty$ and $u$, $v\in W_{0}^{1, m}(\Omega)$ be positive functions satisfying the subsolution and supersolution inequalities
\begin{equation*}
-\Delta_{m} u\leq K(x)u^{-\theta} \quad\mbox{ in }\Omega,
\end{equation*}

\begin{equation*}
-\Delta_{m} v\geq K(x)v^{-\theta} \quad\mbox{ in }\Omega,
\end{equation*}
in the sense of distributions in $W^{-1, \frac{m}{m-1}}(\Omega)$, where $K: \Omega \rightarrow (0, \infty)$ is a positive function and $K\in L_{loc}^{1}(\Omega)$. Then
$$
u\leq v \quad\mbox{ a.e in }\Omega.
$$ 

\end{proposition}

\begin{proposition}\label{uni}
Let $(u, v)$ be a solution of \eq{sys}. Then there exists a constant $c>0$ such that
$$
u(x)\geq c\delta(x) \quad\mbox{ and } v(x)\geq c\delta(x) \quad\mbox{ in }\Omega.
$$
\end{proposition}
\begin{proof}
Let $w$ be the solution of 
\begin{equation*}
\left\{
\begin{aligned}
-\Delta_{m}w&=  C, \;\; w>0 &&\quad\mbox{ in }\Omega,\\
w&= 0 &&\quad\mbox{ on }\partial{\Omega},
\end{aligned}
\right.
\end{equation*}
where $C= ||u||_{\infty}^{-p}||v||_{\infty}^{-q}> 0$. Since $w$ is $m$-superharmonic, by Hopf's Boundary Point lemma we have
$$
\frac{\partial w}{\partial \nu}(x_{0})< 0 \quad\mbox{ on }\partial{\Omega},
$$
which gives 
$$
|\nabla w(x_{0})|> c> 0 \quad\mbox{ on }\partial{\Omega},\quad\mbox{ for some positive constant } c.
$$
Thus,
$$
0<c_{1}< \Big|\frac{\partial w}{\partial \nu}(x_{0})\Big|= \lim_{t\rightarrow 0^{-}}\frac{w(x_{0}+t\nu)}{|t|}.
$$
Hence, there exists $t_{0}< 0$ such that
$$
\frac{w(x_{0}+t\nu)}{|t|}> c_{1} \quad\mbox{ for all } t\in (t_{0}, 0),
$$
which further yields $w(x)\geq c_{1}|t|= c_{1}\delta(x)$. Using the definition of $C$ and $w$ above, we deduce that
$$
-\Delta_{m}u\geq C= -\Delta_{m}w  \quad\mbox{ in }\Omega.
$$
Using the Comparison Principle we obtain 
$$u(x)\geq w(x)\geq c\delta(x) \quad\mbox{ in }\Omega.
$$
In a similar way we prove that $v(x)\geq c\delta(x)$ in $\Omega$. 

\end{proof}

\begin{proposition}\label{tm1}
Let $p\geq 0$ and $K: \Omega \rightarrow (0, \infty)$ be a continuous function such that $K(x)\sim \delta(x)^{-q}$, where $q\geq 0$ and $p\Big(1-\frac{1}{m}\Big)+q< 2-\frac{1}{m}$. Then, the problem 
\begin{equation}\label{eq8a}
\left\{
\begin{aligned}
-\Delta_{m} u&= K(x)u^{-p} &&\quad\mbox{ in }\Omega,\\
u&= 0 &&\quad\mbox{ on }\partial{\Omega},
\end{aligned}
\right.
\end{equation}
has a unique solution $u\in W_{0}^{1, m}(\Omega)$ and:
\begin{itemize}
\item[{\rm (i)}] If $p+q<1$ then $u\in C^{1, \alpha}(\overline{\Omega})$ for some $\alpha \in (0, 1)$ and $u(x)\sim \delta(x)$.
\item[{\rm (ii)}] If $p+q=1$ then $u\in W_{0}^{1, \tau}(\Omega)\cap C^{0, \beta}(\overline{\Omega})$ for some $\beta \in (0, 1)$ and for all $m\leq \tau< \infty$. Also $u\sim \delta(x) \log^{\frac{1}{m+p-1}}\Big(\frac{1}{\delta(x)}\Big)$. 
\item[{\rm (iii)}] If $p+q>1$ then $u\in W_{0}^{1, \tau}(\Omega)\cap C^{0, \beta}(\overline{\Omega})$ for some $\beta \in (0, 1)$ and for all $m\leq \tau< \frac{m+p-1}{p+q-1}$. Also $u\sim \delta(x)^{\frac{m-q}{m+p-1}}$. 
\end{itemize}
\end{proposition}

The key results in proving Proposition \ref{tm1} are the following:
\begin{proposition}\label{p1}
Let $u\in W_{0}^{1, m}(\Omega)\cap C(\overline{\Omega})$, $m>1$ satisfy 
\begin{equation}\label{eq1}
\left\{
\begin{aligned}
-\Delta_{m}u&= \theta(x) &&\quad\mbox{ in }\Omega,\\
u&= 0 &&\quad\mbox{ on }\partial{\Omega},
\end{aligned}
\right.
\end{equation}
where $\theta: \Omega \rightarrow (0, \infty)$ is a continuous function.
\begin{itemize}
\item[{\rm (i)}] If $\theta(x)\sim \delta(x)^{-a}$ for $a\in \Big(1, 2-\frac{1}{m}\Big)$ then $u\in W_{0}^{1, p}(\Omega)$ for all $p\in \Big[m, \frac{m-1}{a-1}\Big)$.
\item[{\rm (ii)}] If $\theta(x)\sim \delta(x)^{-1}\log^{-a}\Big(\frac{1}{\delta(x)}\Big)$ for $a\in (0, 1)$ then $u\in W_{0}^{1, p}(\Omega)$ for all $p\in [m, \infty)$.
\end{itemize}
\end{proposition}

\medskip

\begin{lemma}\label{l1} {\rm (see \cite[Theorem 2]{I1983}). }
Assume $p\geq m> 1$, $u\in W_{0}^{1, m}(\Omega)$ and $\Phi \in L^{\frac{p}{m-1}}(\Omega; \R^{N})$ satisfy 
$$
\Delta_{m}u= {\rm div}(\Phi) \quad\mbox{ in }\Omega.
$$
Then $\nabla u\in L^{p}(\Omega; \R^{N})$ and there exist $c= c(m, p, N)$ such that
$$
||\nabla u||_{L^{p}(\Omega)}^{m-1}\leq c||\Phi||_{L^{\frac{p}{m-1}}(\Omega)}.
$$
\end{lemma}
\begin{lemma}\label{l2}{\rm (see \cite[Lemma 2]{CRT1977} or \cite[Lemma 4.4]{BGH2015}). }
There exist $c>0$ such that if $B_{2r}(x_{0})\subset \Omega$, $0< r\leq 1$ and $v\in W_{0}^{2, p}(\Omega)$, for some $p> N$, then
\begin{equation}
||\nabla v||_{L^{\infty}(B_{r}(x_{0}))}\leq c\Big[r||\Delta v||_{L^{\infty}(B_{2r}(x_{0}))}+ \frac{1}{r}||v||_{L^{\infty}(B_{2r}(x_{0}))}\Big].
\end{equation}
\end{lemma}
\begin{lemma}\label{l3} {\rm (see \cite{LM1991}). }
Let $\Omega \subset \R^{N}$ be a bounded and smooth domain. Then
$$
\int_{\Omega}\delta(x)^{-a}dx< \infty \quad\mbox{ if and only if } a< 1.
$$
\end{lemma}
\begin{proof}[Proof of Proposition \ref{p1}]
Let $w\in C^{2}(\Omega)\cap C(\overline{\Omega})$ satisfy 
\begin{equation}\label{eq2}
\left\{
\begin{aligned}
-\Delta w&= \theta(x) &&\quad\mbox{ in }\Omega,\\
w&=0 &&\quad\mbox{ on }\partial{\Omega}.
\end{aligned}
\right.
\end{equation}
Denote by $\phi>0$ the first eigenfunction of $-\Delta$ in $\Omega$.
\begin{itemize}
\item[{\rm (i)}] Assume $\theta(x)\sim \delta(x)^{-a}$ for some $a\in \Big(1, 2-\frac{1}{m}\Big)$. Then
$$
\underline{w}(x)= \frac{1}{c}\phi(x)^{2-a} \quad\mbox{ and } \quad \overline{w}(x)= c\phi(x)^{2-a}
$$ 
are respectively sub and supersolutions of \eq{eq2} provided $c>1$ is large enough. Hence
\begin{equation}\label{eq3}
w(x)\sim \delta(x)^{2-a}.
\end{equation}
We claim that
\begin{equation}\label{eq4}
|\nabla w(x)| \leq c\delta(x)^{1-a} \quad\mbox{ in }\Omega,
\end{equation}
for some $c>0$. To prove this, let $x\in \Omega$ be a fixed point and $r=\frac{\delta(x)}{3}$. Then 
$$
B_{2r}(x)\subset \Omega_{0}= \{ z\in \Omega: \frac{\delta(x)}{3}\leq \delta(z) \leq \frac{5}{3}\delta(x) \} \subset \Omega
$$
and by Lemma \ref{l2} we have
\begin{equation}\label{eq5}
\begin{aligned}
|\nabla w(x)|\leq c\Big[r||\Delta w||_{L^{\infty}(\Omega_{0})}+ \frac{1}{r}||w||_{L^{\infty}(\Omega_{0})}\Big]\leq c\delta(x)^{1-a}
\end{aligned}
\end{equation}
which proves \eq{eq4}. Using the estimate \eq{eq4} we deduce that $|\nabla w|\in L^{\frac{p}{m-1}}(\Omega)$ whenever $\delta(x)^{1-a}\in L^{\frac{p}{m-1}}(\Omega)$ and by Lemma \ref{l3} this is equivalent to $p< \frac{m-1}{a-1}$. Using Lemma \ref{l1} with $\Phi= \nabla w$ we conclude the proof.
\item[{\rm (ii)}] Assume now $\theta(x)\sim \delta(x)^{-1}\log^{-a}\Big(\frac{1}{\delta(x)}\Big)$ for some $a\in (0, 1)$. Then 
\begin{equation*}
\underline{w}(x)= \frac{1}{c}\phi(x)\log^{1-a}\Big(\frac{A}{\phi(x)}\Big) \quad\mbox{ and } \quad \overline{w}(x)= c\phi(x) \log^{1-a}\Big(\frac{A}{\phi(x)}\Big)
\end{equation*}
are respectively sub and supersolutions of \eq{eq2}, where $A>1$ is large. It follows that
\begin{equation}\label{eq6}
w(x)\sim \delta(x)\log^{1-a}\Big(\frac{A}{\delta(x)}\Big). 
\end{equation}
Using \eq{eq6} and a similar approach as in part(i) we deduce that 
$$
|\nabla w(x)|\leq c\log^{1-a}\Big(\frac{A}{\delta(x)}\Big) \quad\mbox{ in }\Omega,
$$
where $A> 1+{ \rm diam}(\Omega)$. In particular $|\nabla w|\in L^{p}(\Omega)$ for all $p>1$ which, by Lemma \ref{l1} with $\Phi= \nabla w$, yields $u\in W_{0}^{1, p}(\Omega)$ for all $p\in [m, \infty)$. This finishes the proof of Proposition \ref{p1}.
\end{itemize}
\end{proof}

\begin{remark}
The regularity of solution $u$ in Proposition \ref{p1} is optimal. In order to see this, let $(\varphi, \lambda)$ denote the first eigenfunction and eigenvalue of $-\Delta_{m}$ in $\Omega$, that is
\begin{equation}\label{eq7}
\left\{
\begin{aligned}
-\Delta_{m} \varphi&= \lambda |\nabla \varphi|^{m-2}\varphi &&\quad\mbox{ in }\Omega,\\
\varphi&= 0 &&\quad\mbox{ on }\partial{\Omega}.
\end{aligned}
\right.
\end{equation}
It is well known that $\lambda>0$, $\varphi \in C^{1, \gamma}(\overline{\Omega})$ for some $\gamma \in (0, 1)$ and $\varphi$ has constant sign in $\Omega$. Also $\varphi(x) \sim \delta(x)$. Thus, by normalizing $\varphi$ we may assume $\varphi> 0$ in $\Omega$ and $||\varphi||_{\infty}= 1$. To show that the $W_{0}^{1, p}(\Omega)$ regularity in Proposition \ref{p1}(i) is optimal, let $\theta(x)= -\Delta_{m}(\varphi^{\frac{m-a}{m-1}})$. Some straightforward calculations yield 
$$
\theta(x)\sim \varphi(x)^{-a}\sim \delta(x)^{-a}.
$$
Thus, $w= \varphi^{\frac{m-a}{m-1}}$ is a solution of \eq{eq1} with 
$$\theta(x)= -\Delta_{m}(\varphi^{\frac{m-a}{m-1}}).
$$
 Clearly $w\in W_{0}^{1, p}(\Omega)$  for all $m\leq p< \frac{m-1}{a-1}$. By Lemma \ref{l3}, one has $w\not\in W_{0}^{1, \frac{m-1}{a-1}}(\Omega)$.

\medskip

Similarly, to show that the regularity $w\in W_{0}^{1, p}(\Omega)$, $m\leq p< \infty$ is optimal we take 
$$
\theta(x)= -\Delta_{m}\Big(\varphi \log^{1-a}\Big(\frac{A}{\varphi}\Big)\Big),
$$
where $A>1$ is a large constant. 
\end{remark}

\begin{proof}[Proof of Proposition \ref{tm1}]
The existence of a solution $u\in W_{0}^{1, m}(\Omega)$ follows from {\rm \cite[Theorem 3.2]{GST2015}}.
\begin{itemize}
\item[{\rm (i)}] If $p+q<1$, then by {\rm \cite[Theorem 2.1]{GST2015}} we have $u\in C^{1, \alpha}(\overline{\Omega})$, for some $\alpha \in (0, 1)$.
\item[{\rm (ii)}] If $p+q=1$, then by {\rm \cite[Theorem 2.1]{GST2015}} we have $u\in W_{0}^{1, m}(\Omega)\cap C^{0, \beta}(\overline{\Omega})$ for some $\beta \in (0, 1)$. Also, the behaviour $u\sim \delta(x)\log^{\frac{1}{m+p-1}}\Big(\frac{1}{\delta(x)}\Big)$ follows in the same way as in {\rm \cite[Lemma 3.3]{GMS2012}} by noting that
\begin{equation*}
\underline{u}(x)= \frac{1}{c}\varphi(x)\log^{1-a}\Big(\frac{A}{\varphi(x)}\Big) \quad\mbox{ and } \quad \overline{u}(x)= c\varphi(x) \log^{1-a}\Big(\frac{A}{\varphi(x)}\Big)
\end{equation*}
are respectively sub and supersolutions of \eq{eq8a} for some large $c>1$. Using the asymptotic behaviour of $u$ we deduce that 
$$
\theta(x):= K(x)u^{-p}(x)\sim \delta(x)^{-1}\log^{-\frac{p}{m+p-1}}\Big(\frac{1}{\delta(x)}\Big).
$$
By Proposition \ref{p1}(ii) it follows that $u\in W_{0}^{1, \tau}(\Omega)$ for all $\tau \in [m, \infty)$.
\item[{\rm (iii)}] If $p+q>1$, then by {\rm \cite[Theorem 2.1]{GST2015}}, we have $u\in W_{0}^{1, m}(\Omega)\cap C^{0, \beta}(\overline{\Omega})$ for some $\beta \in (0, 1)$. Using the fact that
\begin{equation*}
\underline{u}(x)= \frac{1}{c}\varphi(x)^{\frac{m-q}{m+p-1}} \quad\mbox{ and } \quad\overline{u}(x)= c\varphi(x)^{\frac{m-q}{m+p-1}}
\end{equation*}
are respectively  sub and supersolutions of \eq{eq8a} for some large $c>1$, we easily deduce that 
$$
u\sim \delta(x)^{\frac{m-q}{m+p-1}}.
$$  
Then
$$
\theta(x):= K(x)u^{-p}(x)\sim \delta(x)^{-\frac{mp+(m-1)q}{m+p-1}},
$$
and note that $a= \frac{mp+(m-1)q}{m+p-1}\in \Big(1, 2-\frac{1}{m}\Big)$. By Proposition \ref{p1}(ii) it follows that $u\in W_{0}^{1, \tau}(\Omega)$ for all $\tau \in \Big[m, \frac{m+p-1}{p+q-1}\Big)$.
\end{itemize}
\end{proof}

\section{Proof of Theorem \ref{mr}}
\begin{itemize}
\item[{\rm (i)}] Assume $s-r>1$ and $p+\frac{q(m+r)}{m+s-1}<1$. By Proposition \ref{tm1} ${\rm (i)}$ and ${\rm (iii)}$ there exist $0<c<d<1$ such that:

\medskip

Any subsolution $\underline{v}$ and any supersolution $\overline{v}$ of
\begin{equation}\label{mr1}
\left\{
\begin{aligned}
-\Delta_{m}v&= \delta(x)^{r}v^{-s},\;\; v>0 &&\quad\mbox{ in }\Omega,\\
v&= 0 &&\quad\mbox{ on }\partial{\Omega},
\end{aligned}
\right.
\end{equation}
satisfies
$$
\underline{v}\leq c\delta(x)^{\frac{m+r}{m+s-1}} \quad\mbox{ and } \quad{} \overline{v}\geq d\delta(x)^{\frac{m+r}{m+s-1}} \quad\mbox{ in }\Omega.
$$

\medskip

Any subsolution $\underline{u}$ and any supersolution $\overline{u}$ of
\begin{equation}\label{mr2}
\left\{
\begin{aligned}
-\Delta_{m}u&= \delta(x)^{-\frac{q(m+r)}{m+s-1}}u^{-p},\;\; u>0 &&\quad\mbox{ in }\Omega,\\
u&= 0 &&\quad\mbox{ on }\partial{\Omega},
\end{aligned}
\right.
\end{equation}
satisfies
$$
\underline{u}\leq c\delta(x) \quad\mbox{ and } \quad{} \overline{u}\geq d\delta(x) \quad\mbox{ in }\Omega.
$$

Define 
$$
{\cal A}= \Big\{(u, v) \in C(\overline{\Omega}) \times C(\overline{\Omega}):  \begin{array}{cc} c_{1}\delta(x)\leq u(x) \leq c_{2}\delta(x) &\quad\mbox{ in } \Omega \\ m_{1}\delta(x)^{\frac{m+r}{m+s-1}} \leq v(x) \leq m_{2}\delta(x)^{\frac{m+r}{m+s-1}} &\quad\mbox{ in } \Omega \end{array} \Big\}
$$
where $0< c_{1}< 1< c_{2}$ and $0< m_{1}< 1< m_{2}$ satisfy
\begin{equation}\label{mr3}
d\geq c_{1}m_{2}^{\frac{q}{1+p}}, \quad\mbox{ } c\leq c_{2}m_{1}^{\frac{q}{1+p}},\quad\mbox{ }cc_{2}^{\frac{r}{1+s}}\leq m_{2}, \quad\mbox{ } dc_{1}^{\frac{r}{1+s}}\geq m_{1},
\end{equation}
that is
$$
c_{1}c_{2}^{\frac{qr}{(1+s)(1+p)}}\leq dc^{\frac{-q}{1+p}}\leq cd^{\frac{-q}{1+p}}\leq c_{2}c_{1}^{\frac{qr}{(1+s)(1+p)}}.
$$

For any $(u, v)\in {\cal A}$ let $(Tu, Tv)$ be the unique solution of the system
\begin{equation}\label{mr4}
\left\{
\begin{aligned}
-\Delta_{m} (Tu)&=(Tu)^{-p}v^{-q},Tu>0 &&\quad\mbox{ in } \Omega,\\
-\Delta_{m} (Tv)&=(Tv)^{-s}u^{r}, Tv>0 &&\quad\mbox{ in } \Omega,\\
Tu&=Tv=0 &&\quad\mbox{ on }\partial{\Omega},
\end{aligned}
\right.
\end{equation}
and define 
$$
{\cal F}: {\cal A}\rightarrow C(\overline{\Omega}) \times C(\overline{\Omega}) \quad\mbox{ as } \quad{} {\cal F}(u, v)= (Tu, Tv) \quad\mbox{ for any }(u, v)\in {\cal A}.
$$
In order to prove the existence of a solution to system \eq{sys} we need to show that ${\cal F}$ has a fixed point in ${\cal A}$. We claim that:
\begin{enumerate}
\item[{\rm (a)}] ${\cal F}({\cal A})\subseteq {\cal A}$,
\item[{\rm (b)}] ${\cal F}$ is compact and continuous.
\end{enumerate}

\medskip

Then, by Schauder's fixed point theorem we obtain that ${\cal F}$ has a fixed point in ${\cal A}$ which is further a solution to the system \eq{sys}. 

\medskip

\noindent {\it Step 1: ${\cal F}({\cal A})\subseteq {\cal A}$. } 
Let $(u, v) \in {\cal A}$. From the definition of ${\cal A}$ we have 
$$
v\leq m_{2}\delta(x)^{\frac{m+r}{m+s-1}} \quad\mbox{ in } \Omega.
$$ 
Then $Tu$ satisfies
\begin{equation}\label{mr5}
\left\{
\begin{aligned}
-\Delta_{m}(Tu)&\geq m_{2}^{-q}\delta(x)^{-\frac{q(m+r)}{m+s-1}}(Tu)^{-p},  Tu>0 &&\quad\mbox{ in }\Omega,\\
Tu&=0 &&\quad\mbox{ on }\partial{\Omega}.
\end{aligned}
\right.
\end{equation}
Therefore, $\overline{u}= m_{2}^{\frac{q}{1+p}}Tu$ is a supersolution of \eq{mr2}, which by \eq{mr3} yields
$$
Tu= m_{2}^{-\frac{q}{1+p}}\overline{u}\geq m_{2}^{-\frac{q}{1+p}}d\delta(x)\geq c_{1}\delta(x) \quad\mbox{ in } \Omega.
$$
Also, by the definition of ${\cal A}$
$$
v\geq m_{1}\delta(x)^{\frac{m+r}{m+s-1}} \quad\mbox{ in } \Omega.
$$ 
Thus, by the definition of $Tu$ we deduce that
\begin{equation}\label{mr6}
\left\{
\begin{aligned}
-\Delta_{m}(Tu)&\leq m_{1}^{-q}\delta(x)^{-\frac{q(m+r)}{m+s-1}}(Tu)^{-p},  Tu>0 &&\quad\mbox{ in }\Omega,\\
Tu&=0 &&\quad\mbox{ on }\partial{\Omega}.
\end{aligned}
\right.
\end{equation}
Therefore, $\underline{u}= m_{1}^{\frac{q}{1+p}}Tu$ is a subsolution of \eq{mr2}, which by \eq{mr3} yields
$$
Tu= m_{1}^{-\frac{q}{1+p}}\underline{u}\leq m_{1}^{-\frac{q}{1+p}}c\delta(x)\leq c_{2}\delta(x) \quad\mbox{ in } \Omega.
$$
In a similar manner, by using the definition of ${\cal A}$ and the properties of subsolution and supersolution of problem \eq{mr1}, we obtain that $Tv$ satisfies
\begin{equation*}
m_{1}\delta(x)^{\frac{m+r}{m+s-1}}\leq Tv \leq m_{2}\delta(x)^{\frac{m+r}{m+s-1}} \quad\mbox{ in } \Omega.
\end{equation*}
Hence, $(Tu, Tv) \in {\cal A}$ for all $(u, v) \in {\cal A}$, that is, ${\cal F}({\cal A})\subseteq {\cal A}$.

\medskip

\noindent {\it Step 2: ${\cal F}$ is compact and continuous.} Let $(u, v) \in {\cal A}$. Since ${\cal F}(u, v)= (Tu, Tv) \in {\cal A}$, using Proposition \ref{tm1}, we deduce that
$$
Tu \in C^{1, \alpha}(\overline{\Omega}) \quad\mbox{ and } Tv \in C^{0, \alpha}(\overline{\Omega}),
$$
for some $\alpha \in (0, 1)$. Since the embedding $C^{1, \alpha}(\overline{\Omega}) \hookrightarrow C^{0, \alpha}(\overline{\Omega}) \hookrightarrow C(\overline{\Omega})$ is compact, it follows that ${\cal F}$ is compact.

\medskip

In order to prove that ${\cal F}$ is continuous, let $\{(u_{n}, v_{n})\}\subset {\cal A}$ be such that $u_{n}\rightarrow u$ and $v_{n}\rightarrow v$ in $C(\overline{\Omega})$ as $n \rightarrow \infty$. Since ${\cal F}$ is compact, there exist $(U, V)\in {\cal A}$ such that up to a subsequence we have
$$
Tu_{n}\rightarrow U, \;\;Tv_{n}\rightarrow V \quad\mbox{ in }C(\overline{\Omega}) \quad\mbox{ as }n \rightarrow \infty.
$$
By Theorem 1 in \cite{T1984}, the sequences $\{Tu_{n}\}$ and $\{Tv_{n}\}$ are bounded in $C^{1, \gamma}(\overline{\Omega'})$ 
for any smooth open set $\Omega' \subset \subset \Omega$ with $\gamma \in (0, 1)$. Thus, we have
$$
Tu_{n}\rightarrow U, \;\;Tv_{n}\rightarrow V \quad\mbox{ in }C^{1}(\overline{\Omega'}) \quad\mbox{ as }n \rightarrow \infty,
$$
for any smooth open set $\Omega' \subset \subset \Omega$. Passing to the limit in the definition of $Tu_{n}$ and $Tv_{n}$ we deduce that $(U, V)$ satisfies
\begin{equation}\label{mr7}
\left\{
\begin{aligned}
-\Delta_{m} U&= U^{-p}v^{-q},U>0 &&\quad\mbox{ in } \Omega,\\
-\Delta_{m} V&= V^{-s}u^{r}, V>0 &&\quad\mbox{ in } \Omega,\\
U&= V=0 &&\quad\mbox{ on }\partial{\Omega}.
\end{aligned}
\right.
\end{equation}
Since the solution to \eq{mr4} is unique, it follows that $Tu= U$ and $Tv= V$. Therefore, we deduce that
$$
Tu_{n}\rightarrow Tu, \;\;Tv_{n}\rightarrow Tv \quad\mbox{ in }C(\overline{\Omega}) \quad\mbox{ as }n \rightarrow \infty.
$$
Hence ${\cal F}$ is continuous.

\medskip

Here, we can apply the Schauder's fixed point theorem, there exists $(u, v)\in {\cal A}$ such that ${\cal F}(u, v)= (u, v)$, that is $Tu= u$ and $Tv= v$. Therefore, $(u, v)$ is a solution to the system \eq{sys}.

\medskip

The remaining part of Theorem \ref{mr} is proved in the similar manner. We will only provide the necessary changes in order to carry out the proofs of ${\rm (ii)}$-${\rm (viii)}$. 

\bigskip

\item[{\rm (ii)}] Assume $s-r>1$ and $p+\frac{q(m+r)}{m+s-1}=1$. Choose $\varepsilon> 0$ small enough such that   
$$
s\Big(1-\frac{1}{m}\Big)-r(1-\varepsilon)< 2-\frac{1}{m}.
$$
By Proposition \ref{tm1} ${\rm (i)}$-${\rm (iii)}$ there exist $0<c<d<1$ such that:

\medskip

Any subsolution $\underline{v}$ and any supersolution $\overline{v}$ of
\begin{equation}\label{mr8}
\left\{
\begin{aligned}
-\Delta_{m}v&= \delta(x)^{r}v^{-s},\;\; v>0 &&\quad\mbox{ in }\Omega,\\
v&= 0 &&\quad\mbox{ on }\partial{\Omega},
\end{aligned}
\right.
\end{equation}
satisfies
$$
\underline{v}\leq c\delta(x)^{\frac{m+r}{m+s-1}} \quad\mbox{ and }\quad{} \overline{v}\geq d\delta(x)^{\frac{m+r}{m+s-1}}  \quad\mbox{ in }\Omega.
$$

\medskip

Any subsolution $\underline{u}$ and any supersolution $\overline{u}$ of
\begin{equation}\label{mr9}
\left\{
\begin{aligned}
-\Delta_{m}u&= \delta(x)^{-\frac{q(m+r)}{m+s-1}}u^{-p},\;\; u>0 &&\quad\mbox{ in }\Omega,\\
u&= 0 &&\quad\mbox{ on }\partial{\Omega},
\end{aligned}
\right.
\end{equation}
satisfies
$$
\underline{u}\leq c'\delta(x)\log^{\frac{1}{m+p-1}}\Big(\frac{1}{\delta(x)}\Big)\leq c\delta(x)^{1- \varepsilon} \quad\mbox{ in }\Omega 
$$
and 
$$
\overline{u}\geq d'\delta(x)\log^{\frac{1}{m+p-1}}\Big(\frac{1}{\delta(x)}\Big)\geq d\delta(x) \quad\mbox{ in }\Omega.
$$

Define 
$$
{\cal A}= \Big\{(u, v) \in C(\overline{\Omega}) \times C(\overline{\Omega}):  \begin{array}{cc} c_{1}\delta(x)\leq u(x) \leq c_{2}\delta(x)^{1- \varepsilon} &\quad\mbox{ in } \Omega \\ m_{1}\delta(x)^{\frac{m+r}{m+s-1}} \leq v(x) \leq m_{2}\delta(x)^{\frac{m+r}{m+s-1}} &\quad\mbox{ in } \Omega \end{array} \Big\}
$$
where $0< c_{1}< 1< c_{2}$ and $0< m_{1}< 1< m_{2}$ satisfy \eq{mr3}. The approach is now similar as in part ${\rm (i)}$.

\bigskip

\item[{\rm (iii)}] Assume $s-r<1$ and $p+q<1$. By Proposition \ref{tm1} ${\rm (i)}$-${\rm (iii)}$ there exist $0<c<d<1$ such that:

\medskip

Any subsolution $\underline{v}$ and any supersolution $\overline{v}$ of
\begin{equation}\label{mr15}
\left\{
\begin{aligned}
-\Delta_{m}v&= \delta(x)^{r}v^{-s},\;\; v>0 &&\quad\mbox{ in }\Omega,\\
v&= 0 &&\quad\mbox{ on }\partial{\Omega},
\end{aligned}
\right.
\end{equation}
satisfies
$$
\underline{v}\leq c\delta(x) \quad\mbox{ and } \quad{} \overline{v}\geq d\delta(x) \quad\mbox{ in }\Omega.
$$

\medskip

Any subsolution $\underline{u}$ and any supersolution $\overline{u}$ of
\begin{equation}\label{mr16}
\left\{
\begin{aligned}
-\Delta_{m}u&= \delta(x)^{-q}u^{-p},\;\; u>0 &&\quad\mbox{ in }\Omega,\\
u&= 0 &&\quad\mbox{ on }\partial{\Omega},
\end{aligned}
\right.
\end{equation}
satisfies
$$
\underline{u}\leq c\delta(x) \quad\mbox{ and } \quad{} \overline{u}\geq d\delta(x) \quad\mbox{ in }\Omega.
$$

Define 
$$
{\cal A}= \Big\{(u, v) \in C(\overline{\Omega}) \times C(\overline{\Omega}):  \begin{array}{cc} c_{1}\delta(x)\leq u(x) \leq c_{2}\delta(x) &\quad\mbox{ in } \Omega \\ m_{1}\delta(x) \leq v(x) \leq m_{2}\delta(x) &\quad\mbox{ in } \Omega \end{array} \Big\}
$$
where $0< c_{1}< 1< c_{2}$ and $0< m_{1}< 1< m_{2}$ satisfy \eq{mr3}. The approach is now similar as in part ${\rm (i)}$.

\bigskip

\item[{\rm (iv)}] Assume $s-r<1$ and $p+q=1$. Choose $\varepsilon> 0$ small enough such that   
$$
s\Big(1-\frac{1}{m}\Big)-r(1-\varepsilon)< 2-\frac{1}{m}.
$$
By Proposition \ref{tm1} ${\rm (i)}$-${\rm (iii)}$ there exist $0<c<d<1$ such that:

\medskip

Any subsolution $\underline{v}$ and any supersolution $\overline{v}$ of
\begin{equation}\label{mr22}
\left\{
\begin{aligned}
-\Delta_{m}v&= \delta(x)^{r}v^{-s},\;\; v>0 &&\quad\mbox{ in }\Omega,\\
v&= 0 &&\quad\mbox{ on }\partial{\Omega},
\end{aligned}
\right.
\end{equation}
satisfies
$$
\underline{v}\leq c\delta(x) \quad\mbox{ and } \quad{} \overline{v}\geq d\delta(x) \quad\mbox{ in }\Omega.
$$

\medskip

Any subsolution $\underline{u}$ and any supersolution $\overline{u}$ of
\begin{equation}\label{mr23}
\left\{
\begin{aligned}
-\Delta_{m}u&= \delta(x)^{-q}u^{-p},\;\; u>0 &&\quad\mbox{ in }\Omega,\\
u&= 0 &&\quad\mbox{ on }\partial{\Omega},
\end{aligned}
\right.
\end{equation}
satisfies
$$
\underline{u}\leq c'\delta(x)\log^{\frac{1}{m+p-1}}\Big(\frac{1}{\delta(x)}\Big)\leq c\delta(x)^{1-\varepsilon} \quad\mbox{ in }\Omega
$$
and 
$$
\overline{u}\geq d'\delta(x)log^{\frac{1}{m+p-1}}\Big(\frac{1}{\delta(x)}\Big)\geq d\delta(x) \quad\mbox{ in }\Omega.
$$

Define 
$$
{\cal A}= \Big\{(u, v) \in C(\overline{\Omega}) \times C(\overline{\Omega}):  \begin{array}{cc} c_{1}\delta(x)\leq u(x) \leq c_{2}\delta(x)^{1-\varepsilon} &\quad\mbox{ in } \Omega \\ m_{1}\delta(x) \leq v(x) \leq m_{2}\delta(x) &\quad\mbox{ in } \Omega \end{array} \Big\}
$$
where $0< c_{1}< 1< c_{2}$ and $0< m_{1}< 1< m_{2}$ satisfy \eq{mr3}. The approach is now similar as in part ${\rm (i)}$.

\bigskip

\item[{\rm (v)}] Assume $s-r=1$ and $p+q<1$. Choose $\varepsilon> 0$ small enough such that
$$
p\Big(1-\frac{1}{m}\Big)+q(1-\varepsilon)< 2-\frac{1}{m}.
$$
By Proposition \ref{tm1} ${\rm (i)}$-${\rm (iii)}$ there exist $0<c<d<1$ such that:

\medskip

Any subsolution $\underline{v}$ and any supersolution $\overline{v}$ of
\begin{equation}\label{mr29}
\left\{
\begin{aligned}
-\Delta_{m}v&= \delta(x)^{r}v^{-s},\;\; v>0 &&\quad\mbox{ in }\Omega,\\
v&= 0 &&\quad\mbox{ on }\partial{\Omega},
\end{aligned}
\right.
\end{equation}
satisfies
$$
\underline{v}\leq  c'\delta(x)\log^{\frac{1}{m+s-1}}\Big(\frac{1}{\delta(x)}\Big)\leq c\delta(x)^{1-\varepsilon} \quad\mbox{ in }\Omega
$$
and
$$
\overline{v}\geq  d'\delta(x)\log^{\frac{1}{m+s-1}}\Big(\frac{1}{\delta(x)}\Big)\geq d\delta(x) \quad\mbox{ in }\Omega.
$$

\medskip

Any subsolution $\underline{u}$ and any supersolution $\overline{u}$ of
\begin{equation}\label{mr30}
\left\{
\begin{aligned}
-\Delta_{m}u&= \delta(x)^{-q(1-\varepsilon)}u^{-p},\;\; u>0 &&\quad\mbox{ in }\Omega,\\
u&= 0 &&\quad\mbox{ on }\partial{\Omega},
\end{aligned}
\right.
\end{equation}
satisfies
$$
\underline{u}\leq c\delta(x) \quad\mbox{ and } \quad{} \overline{u}\geq d\delta(x) \quad\mbox{ in }\Omega.
$$

Define 
$$
{\cal A}= \Big\{(u, v) \in C(\overline{\Omega}) \times C(\overline{\Omega}):  \begin{array}{cc} c_{1}\delta(x)\leq u(x) \leq c_{2}\delta(x) &\quad\mbox{ in } \Omega \\ m_{1}\delta(x) \leq v(x) \leq m_{2}\delta(x)^{1-\varepsilon} &\quad\mbox{ in } \Omega \end{array} \Big\}
$$
where $0< c_{1}< 1< c_{2}$ and $0< m_{1}< 1< m_{2}$ satisfy \eq{mr3}. The approach is now similar as in part ${\rm (i)}$.

\bigskip

\item[{\rm (vi)}] Assume $s-r=1$ and $p+q=1$. Choose $\varepsilon, \eta> 0$ small enough such that
$$
p\Big(1-\frac{1}{m}\Big)+q(1-\varepsilon)< 2-\frac{1}{m}
$$
and
$$
s\Big(1-\frac{1}{m}\Big)-r(1-\eta)< 2-\frac{1}{m}.
$$
By Proposition \ref{tm1} ${\rm (i)}$-${\rm (iii)}$ there exist $0<c<d<1$ such that:

\medskip

Any subsolution $\underline{v}$ and any supersolution $\overline{v}$ of
\begin{equation}\label{mr36}
\left\{
\begin{aligned}
-\Delta_{m}v&= \delta(x)^{r}v^{-s},\;\; v>0 &&\quad\mbox{ in }\Omega,\\
v&= 0 &&\quad\mbox{ on }\partial{\Omega},
\end{aligned}
\right.
\end{equation}
satisfies
$$
\underline{v}\leq  c'\delta(x)\log^{\frac{1}{m+s-1}}\Big(\frac{1}{\delta(x)}\Big)\leq c\delta(x)^{1-\varepsilon}  \quad\mbox{ in }\Omega
$$
and 
$$
\overline{v}\geq  d'\delta(x)\log^{\frac{1}{m+s-1}}\Big(\frac{1}{\delta(x)}\Big)\geq d\delta(x) \quad\mbox{ in }\Omega.
$$

\medskip

Any subsolution $\underline{u}$ and any supersolution $\overline{u}$ of
\begin{equation}\label{mr37}
\left\{
\begin{aligned}
-\Delta_{m}u&= \delta(x)^{-q(1-\varepsilon)}u^{-p},\;\; u>0 &&\quad\mbox{ in }\Omega,\\
u&= 0 &&\quad\mbox{ on }\partial{\Omega},
\end{aligned}
\right.
\end{equation}
satisfies
$$
\underline{u}\leq  c'\delta(x)\log^{\frac{1}{m+p-1}}\Big(\frac{1}{\delta(x)}\Big)\leq c\delta(x)^{1-\eta} \quad\mbox{ in }\Omega
$$
and 
$$
\overline{u}\geq  d'\delta(x)\log^{\frac{1}{m+p-1}}\Big(\frac{1}{\delta(x)}\Big)\geq d\delta(x) \quad\mbox{ in }\Omega.
$$

Define 
$$
{\cal A}= \Big\{(u, v) \in C(\overline{\Omega}) \times C(\overline{\Omega}):  \begin{array}{cc} c_{1}\delta(x)\leq u(x) \leq c_{2}\delta(x)^{1-\eta} &\quad\mbox{ in } \Omega \\ m_{1}\delta(x) \leq v(x) \leq m_{2}\delta(x)^{1-\varepsilon} &\quad\mbox{ in } \Omega \end{array} \Big\}
$$
where $0< c_{1}< 1< c_{2}$ and $0< m_{1}< 1< m_{2}$ satisfy \eq{mr3}. The approach is now similar as in part ${\rm (i)}$.

\bigskip

\item[{\rm (vii)}] Assume $p+q>1$ and $s-\frac{r(m-q)}{m+p-1}<1$. By Proposition \ref{tm1} ${\rm (i)}$-${\rm (iii)}$ there exist $0<c<d<1$ such that:

\medskip

Any subsolution $\underline{u}$ and any supersolution $\overline{u}$ of
\begin{equation}\label{mr43}
\left\{
\begin{aligned}
-\Delta_{m}u&= \delta(x)^{-q}u^{-p},\;\; u>0 &&\quad\mbox{ in }\Omega,\\
u&= 0 &&\quad\mbox{ on }\partial{\Omega},
\end{aligned}
\right.
\end{equation}
satisfies
$$
\underline{u}\leq c\delta(x)^{\frac{m-q}{m+p-1}} \quad\mbox{ and } \quad{} \overline{u}\geq d\delta(x)^{\frac{m-q}{m+p-1}}  \quad\mbox{ in }\Omega.
$$

\medskip

Any subsolution $\underline{v}$ and any supersolution $\overline{v}$ of
\begin{equation}\label{mr44}
\left\{
\begin{aligned}
-\Delta_{m}v&= \delta(x)^{\frac{r(m-q)}{m+p-1}}v^{-s},\;\; v>0 &&\quad\mbox{ in }\Omega,\\
v&= 0 &&\quad\mbox{ on }\partial{\Omega},
\end{aligned}
\right.
\end{equation}
satisfies
$$
\underline{v}\leq c\delta(x) \quad\mbox{ and } \quad{} \overline{v}\geq d\delta(x) \quad\mbox{ in }\Omega.
$$

Define 
$$
{\cal A}= \Big\{(u, v) \in C(\overline{\Omega}) \times C(\overline{\Omega}):  \begin{array}{cc} c_{1}\delta(x)^{\frac{m-q}{m+p-1}}\leq u(x) \leq c_{2}\delta(x)^{\frac{m-q}{m+p-1}} &\quad\mbox{ in } \Omega \\ m_{1}\delta(x)\leq v(x) \leq m_{2}\delta(x) &\quad\mbox{ in } \Omega \end{array} \Big\}
$$
where $0< c_{1}< 1< c_{2}$ and $0< m_{1}< 1< m_{2}$ satisfy \eq{mr3}. The approach is now similar as in part ${\rm (i)}$.

\bigskip

\item[{\rm (viii)}] Assume $p+q>1$ and $s-\frac{r(m-q)}{m+p-1}=1$. Choose $\varepsilon> 0$ such that
$$
p\Big(1-\frac{1}{m}\Big)+q(1-\varepsilon)< 2-\frac{1}{m}.
$$
By Proposition \ref{tm1} ${\rm (i)}$-${\rm (iii)}$ there exist $0<c<d<1$ such that:

\medskip

Any subsolution $\underline{u}$ and any supersolution $\overline{u}$ of
\begin{equation}\label{mr50}
\left\{
\begin{aligned}
-\Delta_{m}u&= \delta(x)^{-q}u^{-p},\;\; u>0 &&\quad\mbox{ in }\Omega,\\
u&= 0 &&\quad\mbox{ on }\partial{\Omega},
\end{aligned}
\right.
\end{equation}
satisfies
$$
\underline{u}\leq c\delta(x)^{\frac{m-q}{m+p-1}} \quad\mbox{ and } \quad{} \overline{u}\geq d\delta(x)^{\frac{m-q}{m+p-1}} \quad\mbox{ in }\Omega.
$$

\medskip

Any subsolution $\underline{v}$ and any supersolution $\overline{v}$ of
\begin{equation}\label{mr51}
\left\{
\begin{aligned}
-\Delta_{m}v&= \delta(x)^{\frac{r(m-q)}{m+p-1}}v^{-s},\;\; v>0 &&\quad\mbox{ in }\Omega,\\
v&= 0 &&\quad\mbox{ on }\partial{\Omega},
\end{aligned}
\right.
\end{equation}
satisfies
$$
\underline{v}\leq c'\delta(x)\log^{\frac{1}{m+s-1}}\Big(\frac{1}{\delta(x)}\Big)\leq c\delta(x)^{1-\varepsilon}  \quad\mbox{ in }\Omega
$$
and
$$
\overline{v}\geq d'\delta(x)\log^{\frac{1}{m+s-1}}\Big(\frac{1}{\delta(x)}\Big)\geq d\delta(x) \quad\mbox{ in }\Omega.
$$

Define 
$$
{\cal A}= \Big\{(u, v) \in C(\overline{\Omega}) \times C(\overline{\Omega}):  \begin{array}{cc} c_{1}\delta(x)^{\frac{m-q}{m+p-1}}\leq u(x) \leq c_{2}\delta(x)^{\frac{m-q}{m+p-1}} &\quad\mbox{ in } \Omega \\ m_{1}\delta(x)\leq v(x) \leq m_{2}\delta(x)^{1-\varepsilon} &\quad\mbox{ in } \Omega \end{array} \Big\}
$$
where $0< c_{1}< 1< c_{2}$ and $0< m_{1}< 1< m_{2}$ satisfy \eq{mr3}. The approach is now similar as in part ${\rm (i)}$.

\end{itemize}

\medskip

\section{ Proof of Theorem \ref{mru} }

\begin{itemize}
\item[{\rm (i)}] Assume $s-r>1$ and $p+\frac{q(m+r)}{m+s-1}< 1$. Let $(u_{1}, v_{1})$ and $(u_{2}, v_{2})$ be two solutions of system \eq{sys}. By Proposition \ref{uni}, there exists a constant $c_{1}> 0$ such that
\begin{equation}\label{uni1}
u_{i}(x), \;\; v_{i}(x)\geq c_{1}\delta(x) \quad\mbox{ in }\Omega, \mbox{ } i= 1,2.
\end{equation}
Using $u_{i}(x)\geq c_{1}\delta(x)$ in $\Omega$, we deduce that
\begin{equation*}
\left\{
\begin{aligned}
-\Delta_{m}v_{i}&\geq c_{2}\delta(x)^{r}v_{i}^{-s},\;\; u_{i}>0 &&\quad\mbox{ in }\Omega,\\
v_{i}&= 0 &&\quad\mbox{ on }\partial{\Omega},
\end{aligned}
\right.
\end{equation*}
satisfies
\begin{equation}\label{cu1}
v_{i}(x)\geq c_{3}\delta(x)^{\frac{m+r}{m+s-1}} \quad\mbox{ in }\Omega, \mbox{ } i= 1,2.
\end{equation}
Using \eq{cu1} we deduce that $u_{i}$ satisfies 
\begin{equation*}
\left\{
\begin{aligned}
-\Delta_{m}u_{i}&\leq c_{4}\delta(x)^{-\frac{q(m+r)}{m+s-1}}u_{i}^{-p},\;\; u_{i}>0 &&\quad\mbox{ in }\Omega,\\
u_{i}&= 0 &&\quad\mbox{ on }\partial{\Omega},
\end{aligned}
\right.
\end{equation*}
so
\begin{equation}\label{cu2}
u_{i}\leq c_{5}\delta(x) \quad\mbox{ in }\Omega, \mbox{ } i= 1,2.
\end{equation}
By \eq{uni1} and \eq{cu2}, there exists $c\in (0, 1)$ such that
$$
c\delta(x)\leq u_{i}(x)\leq \frac{1}{c}\delta(x), \quad\mbox{ in }\Omega \mbox{ } i= 1,2.
$$
Therefore, we can find a constant $C>1$ such that
$$
Cu_{1}\geq u_{2} \quad\mbox{ and } \quad{} Cu_{2}\geq u_{1} \quad\mbox{ in }\Omega.
$$
We claim that $u_{1}\geq u_{2}$ in $\Omega$. Let us suppose by contradiction that
$$
M:= \inf\{A>1 : Au_{1}\geq u_{2} \mbox{ in }\Omega\}>1.
$$
Clearly $Mu_{1}\geq u_{2}$ in $\Omega$. Therefore, it follows that
$$
-\Delta_{m}v_{2}\leq M^{r}u_{1}^{r}v_{2}^{-s} \quad\mbox{ in }\Omega.
$$
Thus, $v_{1}$ is a solution and $M^{-\frac{r}{1+s}}v_{2}$ is subsolution of
\begin{equation*}
\left\{
\begin{aligned}
-\Delta_{m}z&= u_{1}^{r}z^{-s},\;\; z>0 &&\quad\mbox{ in }\Omega,\\
z&= 0, &&\quad\mbox{ on }\partial{\Omega},
\end{aligned}
\right.
\end{equation*}
which by Proposition \ref{gst} gives 
$$
M^{-\frac{r}{1+s}}v_{2}\leq v_{1} \quad\mbox{ in }\Omega.
$$
Using the above inequality, we have 
\begin{equation*} 
-\Delta_{m}u_{1}\leq M^{\frac{qr}{1+s}}u_{1}^{-p}v_{2}^{-q} \quad\mbox{ in }\Omega.
\end{equation*}
Thus, $u_{2}$ is a solution and $M^{-\frac{qr}{(1+p)(1+s)}}u_{1}$ is subsolution of
\begin{equation*}
\left\{
\begin{aligned}
-\Delta_{m}w&= w^{-p}v_{2}^{-q},\;\; w>0 &&\quad\mbox{ in }\Omega,\\
w&= 0, &&\quad\mbox{ on }\partial{\Omega},
\end{aligned}
\right.
\end{equation*}
which by Proposition \ref{gst} gives 
$$
u_{1}\leq M^{\frac{qr}{(1+p)(1+s)}}u_{2} \quad\mbox{ in }\Omega.
$$
This contradicts the minimality of $M$ as 
$$
\frac{qr}{(1+p)(1+s)}< 1.
$$
Hence $u_{1}\geq u_{2}$ in $\Omega$. Similarly we have $u_{2}\geq u_{1}$ in $\Omega$, so $u_{1}\equiv u_{2}$ which also yields $v_{1}\equiv v_{2}$. Thus, the system \eq{sys} has unique solution.

\medskip

\item[{\rm (ii)}] Assume $s-r\leq 1$ and $p+q<1$.
Let $(u_{1}, v_{1})$ and $(u_{2}, v_{2})$ be two solutions of \eq{sys}. By Proposition \ref{uni}, there exists a constant $c_{1}> 0$ such that
\begin{equation}\label{uni2}
u_{i}(x), \;\; v_{i}(x)\geq c_{1}\delta(x) \quad\mbox{ in }\Omega, \mbox{ } i= 1,2.
\end{equation}
Using $v_{i}(x)\geq c_{1}\delta(x)$ in $\Omega$, we deduce that $u_{i}$ satisfies
\begin{equation*}
\left\{
\begin{aligned}
-\Delta_{m}u_{i}&\leq c_{2}\delta(x)^{-q}u_{i}^{-p},\;\; u_{i}>0 &&\quad\mbox{ in }\Omega,\\
u_{i}&= 0 &&\quad\mbox{ on }\partial{\Omega},
\end{aligned}
\right.
\end{equation*}
so
\begin{equation}\label{cub}
u_{i}(x)\leq c_{3}\delta(x) \quad\mbox{ in }\Omega, \mbox{ } i= 1,2.
\end{equation}
By \eq{uni2} and \eq{cub}, there exists $c\in (0, 1)$ such that
$$
c\delta(x)\leq u_{i}(x)\leq \frac{1}{c}\delta(x) \quad\mbox{ in }\Omega, \mbox{ } i= 1,2.
$$
Further, the approach is similar to that in case ${\rm (i)}$.

\medskip

\item[{\rm (iii)}] Assume $p+q=1$ and $s-r< 1$. Let $\varepsilon>0$ be small enoungh such that $s-r(1-\varepsilon)<1$. Let $(u_{1}, v_{1})$ and $(u_{2}, v_{2})$ be two solutions of \eq{sys}. By Proposition \ref{uni}, there exists a constant $c_{1}> 0$ such that
\begin{equation}\label{unine}
u_{i}(x), \;\; v_{i}(x)\geq c_{1}\delta(x) \quad\mbox{ in }\Omega, \mbox{ } i= 1,2.
\end{equation}
Using $v_{i}(x)\geq c_{1}\delta(x)$ in $\Omega$, we deduce that $u_{i}$ satisfies
\begin{equation*}
\left\{
\begin{aligned}
-\Delta_{m}u_{i}&\leq c_{2}\delta(x)^{-q}u_{i}^{-p},\;\; u_{i}>0 &&\quad\mbox{ in }\Omega,\\
u_{i}&= 0 &&\quad\mbox{ on }\partial{\Omega},
\end{aligned}
\right.
\end{equation*}
so
\begin{equation}\label{unine2}
u_{i}(x)\leq c_{3}\delta(x)\log^{\frac{1}{m+p-1}}\Big(\frac{1}{\delta(x)}\Big)\leq c_{4}\delta(x)^{1-\varepsilon} \quad\mbox{ in }\Omega, \mbox{ } i= 1,2.
\end{equation}
Using \eq{unine2} we deduce that $v_{i}$ satisfies 
\begin{equation*}
\left\{
\begin{aligned}
-\Delta_{m}v_{i}&\leq c_{5}\delta(x)^{r(1-\varepsilon)}v_{i}^{-s},\;\; v_{i}>0 &&\quad\mbox{ in }\Omega,\\
v_{i}&= 0 &&\quad\mbox{ on }\partial{\Omega},
\end{aligned}
\right.
\end{equation*}
so
\begin{equation}\label{unine3}
v_{i}\leq c_{6}\delta(x) \quad\mbox{ in }\Omega, \mbox{ } i= 1,2.
\end{equation}
By \eq{unine} and \eq{unine3}, there exists $c\in (0, 1)$ such that
\begin{equation}\label{unine4}
c\delta(x)\leq v_{i}(x)\leq \frac{1}{c}\delta(x) \quad\mbox{ in }\Omega, \mbox{ } i= 1,2.
\end{equation}
Further, the approach is similar to that in case ${\rm (i)}$.

\medskip

\item[{\rm (iv)}] Assume $p+q>1$ and $s-\frac{r(m-q)}{m+p-1}< 1$. Let $(u_{1}, v_{1})$ and $(u_{2}, v_{2})$ be two solutions of system \eq{sys}. By Proposition \ref{uni}, there exists a constant $c_{1}> 0$ such that
\begin{equation}\label{uni3}
u_{i}(x), \;\; v_{i}(x)\geq c_{1}\delta(x) \quad\mbox{ in }\Omega, \mbox{ } i= 1,2.
\end{equation}
Using $v_{i}(x)\geq c_{1}\delta(x)$ in $\Omega$, we deduce that $u_{i}$ satisfies
\begin{equation*}
\left\{
\begin{aligned}
-\Delta_{m}u_{i}&\leq c_{2}\delta(x)^{-q}u_{i}^{-p},\;\; u_{i}>0 &&\quad\mbox{ in }\Omega,\\
u_{i}&= 0 &&\quad\mbox{ on }\partial{\Omega},
\end{aligned}
\right.
\end{equation*}
so
\begin{equation}\label{cu3}
u_{i}(x)\leq c_{3}\delta(x)^{\frac{m-q}{m+p-1}} \quad\mbox{ in }\Omega,\mbox{ } i= 1,2.
\end{equation}
Using \eq{cu3} we deduce that $v_{i}$ satisfies
\begin{equation*}
\left\{
\begin{aligned}
-\Delta_{m}v_{i}&\leq c_{4}\delta(x)^{\frac{r(m-q)}{m+p-1}}v_{i}^{-s},\;\; v_{i}>0 &&\quad\mbox{ in }\Omega,\\
v_{i}&= 0 &&\quad\mbox{ on }\partial{\Omega},
\end{aligned}
\right.
\end{equation*}
so
\begin{equation}\label{cu4}
v_{i}\leq c_{5}\delta(x) \quad\mbox{ in }\Omega, \mbox{ } i= 1,2.
\end{equation}
By \eq{uni3} and \eq{cu4}, there exists $c\in (0, 1)$ such that
$$
c\delta(x)\leq v_{i}(x)\leq \frac{1}{c}\delta(x), \quad\mbox{ in }\Omega, \mbox{ } i= 1,2.
$$
Therefore, we can find a constant $C>1$ such that
$$
Cv_{1}\geq v_{2} \quad\mbox{ and } \quad{} Cv_{2}\geq v_{1} \quad\mbox{ in }\Omega.
$$
From now on we proceed in the same manner as in case ${\rm (i)}$. This concludes the proof of Theorem \ref{mru}.

\end{itemize}

\noindent{\bf Acknowledgement.} This work is part of the author's PhD thesis and has been carried out with the financial support of the Research Demonstratorship Scheme offered by the School of Mathematics and Statistics, University College Dublin.

\end{document}